\documentclass[12pt]{amsart}
\usepackage[margin=1.2in]{geometry}
\usepackage{graphicx,latexsym}
\usepackage{comment}
\usepackage{amsfonts, amssymb, amsmath, amsthm, bm, hyperref}

\newtheorem{theorem}{Theorem}

\newtheorem{corollary}{Corollary}
\theoremstyle{definition}

\theoremstyle{remark}

\begin{document}
\title[Vector valued Hardy spaces and distributional boundary values]{Note on vector valued Hardy spaces related to analytic functions having distributional boundary values}

\author[R. D. Carmichael]{Richard D. Carmichael}

\address{R. D. Carmichael \\ Department of Mathematics and Statistics\\ Wake Forest University\\ Winston-Salem\\ N.C. 27109-7388 \\
U.S.A.}
\email{carmicha@wfu.edu}

\author[S. Pilipovi\'{c}]{Stevan Pilipovi\'{c}}
\address{S. Pilipovi\'{c}\\ Department of Mathematics and Informatics\\
University of Novi Sad\\
Trg Dositeja Obradovi\'{c}a 4\\
21000 Novi Sad\\
Serbia}
\email{Stevan.Pilipovic@dmi.uns.ac.rs}

\author[J. Vindas]{Jasson Vindas}
\thanks {J. Vindas was supported by Ghent University through the BOF-grants 01J11615 and 01J04017.}
\address{J. Vindas, Department of Mathematics: Analysis, Logic and Discrete Mathematics\\ Ghent University\\ Krijgslaan 281\\ 9000 Gent\\ Belgium}
\email{jasson.vindas@UGent.be}

\subjclass[2010]{32A07; 32A35; 46F20; 32A26.}
\keywords{Hardy spaces; vector valued analytic functions on tube domains; vector valued distributional boundary values; Poisson integral transform.}
\begin{abstract}
Analytic functions defined on a tube domain $T^{C}\subset \mathbb{C}^{n}$ and taking values in a Banach space $X$ which are known to have $X$-valued distributional boundary values are shown to be in the Hardy space $H^{p}(T^{C},X)$ if the boundary value is in the vector valued Lebesgue space $L^{p}(\mathbb{R}^{n},X)$, where $1\leq p \leq \infty$ and $C$ is a regular open convex cone. Poisson integral transform representations of elements of $H^{p}(T^{C}, X)$ are also obtained for certain classes of Banach spaces, including reflexive Banach spaces.
\end{abstract}

\maketitle

\section{Introduction}
 In \cite{C-R83} Carmichael and Richters proved that if a holomorphic function on a tube domain having as base a regular open convex cone has an $L^p$ function (with $1\leq p\leq \infty$) as distributional boundary value, then the holomorphic function should belong to the Hardy space $H^{p}$ on the tube. The authors have recently obtained a vector valued generalization of this result in \cite{C-P-V2019}. However, we were only able to prove the desired vector valued version in the range $2\leq p\leq \infty$,  and only for Hilbert space valued spaces; see \cite[Theorem 4.4]{C-P-V2019}.
 
The aim of this note is to improve the quoted main result from \cite{C-P-V2019} by showing that it holds for any $1\leq p\leq \infty $ and any Banach space. This will be done in Section \ref{bvHardy section 3}. Further, in Section \ref{bvHardy section 4}, we prove that any element of an $X$-valued Hardy space is representable as a Poisson integral if $X$ is a dual Banach space satisfying the Radon-Nikod\'{y}m property. In particular, the latter holds for reflexive Banach spaces.

Distributional boundary value results associated with Hardy spaces have been of importance in particle physics; see \cite{Raina} for example. The distributional boundary value result of \cite{Raina} motivated the authors' work in \cite{C-P-V2019} and the current paper.
\section{Notation}
Let $X$ be a Banach space with norm $\|\:\cdot\:\|_{X}$. Integrals for $X$-valued functions are interpreted in the Bochner sense \cite{D-U77,ryan2002} and the $X$-valued Lebesgue spaces $L^{p}(\mathbb{R}^{n}, X)$, $p\in (0, \infty]$, are defined in the usual way. The space of $X$-valued distributions \cite{T67} is the space of continuous linear mappings $\mathcal{D}(\mathbb{R}^{n})\to X$, denoted as $\mathcal{D}'(\mathbb{R}^{n}, X)$. In analogy to the scalar valued case, we denote the evaluation of a vector valued distribution $\mathbf{f}\in \mathcal{D}'(\mathbb{R}^{n},X)$ at a test function $\varphi\in \mathcal{D}(\mathbb{R}^{n})$ as $\langle \mathbf{f}, \varphi \rangle\in X$.

An open convex cone (with vertex at the origin) $C\subset \mathbb{R}^{n} $ is called \emph{regular} if its closure does not contain any entire straight line. Equivalently, regularity means that the conjugate cone $C^{\ast}=\{y\in\mathbb{R}^{n}:\: y\cdot x\geq 0, \:\forall x\in C\}$ has non-empty interior. The tube domain with base $C$ is $T^{C}=\mathbb{R}^{n}+iC$. The Cauchy-Szeg\"{o} kernel of $T^{C}$ is defined as
$
K(z) =\int_{C^{*}} e^{2 \pi i z \cdot u} \mathrm{d} u $ for  $z \in T^{C},$
while its corresponding Poisson kernel is
\[
Q(z;u) = \frac{|K(z-u)|^{2}}{K(2iy)}, \qquad u \in \mathbb{R}^{n}, \; z=x+iy \in T^{C}.
\]
It should be noted that $Q(z; \cdot)\in L^{p}(\mathbb{R}^{n})$ for any $p\in[1,\infty]$; see \cite[3.7, p.~105]{S-W71}.

If $C$ is an open cone and $0<p\leq \infty$, the $X$-valued Hardy space consists those vector valued holomorphic functions $\mathbf{F}:T^{C}\to X$ such that
\[
\sup_{y\in C}\int_{\mathbb{R}^{n}} \|\textbf{F}(x+iy))\|_{X}^{p} \mathrm{d} x <\infty,
\]
where the usual modification is made for the case $p = \infty$.

\section{Distributional boundary values in $L^{p}(\mathbb{R}^{n},X)$}
\label{bvHardy section 3}
In this section we improve \cite[Theorem 4.4, p.~1650]{C-P-V2019}. It is worth pointing out that our method here is much simpler and shorter than the one employed in \cite{C-P-V2019}.
\begin{theorem} \label{bvHardyTh1}Let $X$ be a Banach spaces, let $C$ be a regular open convex cone, and let $p\in[1,\infty].$
Suppose that the vector valued holomorphic function $\mathbf{F}: T^{C}\to X$ has distribution boundary value $\mathbf{f} \in L^{p}(\mathbb{R}^{n}, X)$, that is,
\begin{equation}
\label{bvHardyeq1}
 \underset{y\in C}{\lim_ {y\to 0}} \int_{\mathbb{R}^{n}} \mathbf{F}(x+iy) \varphi(x)\mathrm{d} x= \int_{\mathbb{R}^{n}}\mathbf{f}(x) \varphi(x)\mathrm{d} x  \qquad \mbox{in } X
\end{equation} 
holds for each test function $\varphi \in \mathcal{D}(\mathbb{R}^{n})$. Then,  $\mathbf{F} \in H^{p}(T^{C},X)$ and
\begin{equation}\label{bvHardyeq2}
\mathbf{F}(z) = \int_{\mathbb{R}^{n}} \mathbf{f}(u) Q(z;u) \mathrm{d} u , \qquad z \in T^{C}.
\end{equation}
\end{theorem}
\begin{proof}
Define 
$$
\mathbf{G}(z)= \int_{\mathbb{R}^{n}} \mathbf{f}(u) Q(z;u) \mathrm{d} u , \qquad z \in T^{C}.
$$
We have shown in \cite[Lemma 3.4, p.~1639]{C-P-V2019} that $\mathbf{G}\in H^{p}(T^{C}, X)$. Moreover, the quoted lemma also gives that $\mathbf{G}$ has distributional boundary value $\mathbf{f}$. It thus suffices to prove that $\mathbf{F}(z)=\mathbf{G}(z)$, $z\in T^{C}$, which we verify via the Hanh-Banach theorem and the (scalar valued) edge-of-the-wedge-theorem. Let $\mathbf{w}^{\ast}\in X'$. Consider the scalar valued holomorphic function $H_{\mathbf{w}^{\ast}}(z)=\langle\mathbf{w}^{\ast}, \mathbf{F}(z)-\mathbf{G}(z)\rangle$. It satisfies 
\[
\underset{y\in C}{\lim_ {y\to 0}} \int_{\mathbb{R}^{n}} H_{\mathbf{w}^{\ast}}(x+iy) \varphi(x)\mathrm{d} x= \left\langle \mathbf{w}^{\ast} ,  \underset{y\in C}{\lim_ {y\to 0}} \int_{\mathbb{R}^{n}} (\mathbf{F}(x+iy)-\mathbf{G}(x+iy)) \varphi(x)\mathrm{d} x \right\rangle = 0,
\]
for each $\varphi\in \mathcal{D}(\mathbb{R}^{n})$. Using \cite[Corollary of Theorem B, p.~20]{Rudin71}, we obtain $\langle\mathbf{w}^{\ast}, \mathbf{F}(z)-\mathbf{G}(z)\rangle=0$ for all $z\in T^{C}$. Since $\mathbf{w}^{\ast}\in X'$ was arbitrary, the Hanh-Banach theorem yields the equality $\mathbf{F}(z)=\mathbf{G}(z)$, $z\in T^{C}$. This establishes the theorem.
\end{proof}

\section{Poisson integral representation}

\label{bvHardy section 4}

We now turn our attention to Banach spaces $X$ where the Poisson integral representation \eqref{bvHardyeq2}, for some $\mathbf{f}\in L^{p}(\mathbb{R}^{n},X)$, is valid for any $X$-valued holomorphic function $\mathbf{F}$ belonging to the Hardy space $H^{p}(T^{C}, X)$.

We need to introduce some terminology. A Banach space is said to have the \emph{Radon-Nikod\'{y}m property} if the Radon-Nikod\'{y}m theorem holds for vector measures on it; see  \cite[p.~61]{D-U77} or \cite[Chapter 5, p.~102]{ryan2002} for precise definitions and background on material Banach spaces with this property. We call $X$ a \emph{dual Banach space} if it is the strong dual of some Banach space.

\begin{theorem}\label{bvHardyTh2} Let $X$ be a dual Banach space having the 
Radon-Nikod\'{y}m property, let $C$ be a regular open convex  cone, and let $p\in[1,\infty]$. If $\mathbf{F}\in H^{p}(T^{C},X)$, then there is $\mathbf{f}\in L^{p}(\mathbb{R}^{n}, X)$ such that $\mathbf{F}$ has the Poisson integral representation \eqref{bvHardyeq2}.
\end{theorem}
\begin{proof} We first show that $\mathbf{F}$ has distributional boundary value. Let $\varphi\in\mathcal{D}(\mathbb{R}^{n})$ and write $\Phi(x,y)=\varphi(x)+i\sum_{j=1}^{n}  y_j \partial_{j} \varphi(x)$ for $x\in \mathbb{R}^{n}$ and $y=(y_1,\dots, y_n)\in T^{C}$. Pick a unit vector $\omega \in C$.  Applying the Stokes theorem exactly in the same way as in \cite[p. 67]{HormanderVolI}, we have, if $y\in C$,
\begin{align*}
\int_{\mathbb{R}^{n}}\mathbf{F}(x +iy)\varphi(x)\mathrm{d} x &= \int_{\mathbb{R}^{n}}\mathbf{F}(x +iy+i\omega)\Phi(x,\omega)\mathrm{d} x
\\
&\qquad +i\sum_{j=1}^{n}\omega_j \int_{0}^{1}\int_{\mathbb{R}^{n}}\mathbf{F}(x+i t\omega+iy)\partial_{j}\varphi (x)\mathrm{d} x\mathrm{d} t
\end{align*}
Since $\mathbf{F}\in H^{p}(T^{C}, X)$, we may take the limit as $y\to 0$ in the right-hand side of the above expression and conclude that $\mathbf{F}$ has distributional boundary value $\mathbf{f}\in \mathcal{D}'(\mathbb{R}^{n}, X)$, given in fact as
\[
\langle \mathbf{f},\varphi\rangle= \int_{\mathbb{R}^{n}}\mathbf{F}(x+i\omega)\Phi(x,\omega)\mathrm{d} x +i\sum_{j=1}^{n}\omega_j \int_{0}^{1}\int_{\mathbb{R}^{n}}\mathbf{F}(x+i t\omega)\partial_{j}\varphi (x)\mathrm{d} x\mathrm{d}t.
\]
In view of Theorem \ref{bvHardyTh1}, the representation \eqref{bvHardyeq2} would follow at once if we are able to show that $\mathbf{f}\in L^{p}(\mathbb{R}^{n},X)$. We now focus in showing the latter.
The rest of the proof exploits the fact that we can consider a weak* topology on $L^{p}(\mathbb{R}^{n},X)$ due to our assumptions on $X$. Let $Y$ be a Banach space such that $X=Y'$. For each $y\in C$, write $\mathbf{F}_{y}(x)=\mathbf{F}(x+iy)$. We split our considerations in two cases.

\emph{Case I: $1<p\leq \infty$.} In this case it is well-known (see\footnote{This result is stated in \cite[p.~98]{D-U77} for vector valued $L^{p}$-spaces with respect to finite (scalar valued positive) measures, but the proof given there shows that it holds for $\sigma$-finite measures, in particular for the Lebesgue measure as we used it here.} \cite[Theorem~1, Sect.~IV.1, p.~98]{D-U77}) that $L^{p}(\mathbb{R}^{n}, X)$ is the strong dual of $L^{q}(\mathbb{R}^{n}, Y)$ where $1/p+1/q=1$. Besides its strong topology, we also provide $L^{p}(\mathbb{R}^{n}, X)$  with the weak* topology with respect to this duality. Since the membership $\mathbf{F}\in H^{p}(T^{C},X)$ precisely means that the set $\{\mathbf{F}_{y}:\: y\in C\}$ is strongly bounded in  $L^{p}(\mathbb{R}^{n}, X)$, the Banach-Alaoglu theorem \cite{T67} yields the existence of a sequence of points $y_k\in C$ and an $X$-valued function $\mathbf{g}\in L^{p}(\mathbb{R}^{n}, X)$ such that 
$
\mathbf{F}_{y_k}\to \mathbf{g}$ as $k\to\infty$, weakly* in  $L^{p}(\mathbb{R}^{n}, X)$. But, this weak* convergence is stronger than convergence in $\mathcal{D}'(\mathbb{R}^{n},X)$, whence $\mathbf{f}=\mathbf{g}\in L^{p}(\mathbb{R}^{n}, X)$, as required. 

\emph{Case II: $p=1$}. Denote as $\mathcal{M}_{1}(\mathbb{R}^{n}, X)$ the Banach space of $X$-valued vector measures with finite variation \cite[Chapter 5]{ryan2002} (cf. \cite{D-U77}) on the $\sigma$-algebra of Borel sets of $\mathbb{R}^{n}$. We regard $L^{1}(\mathbb{R}^{n}, X)$ as a closed subspace of $\mathcal{M}_{1}(\mathbb{R}^{n}, X)$. Let $\mathcal{M}(\mathbb{R}^{n})$ be the space of (signed) Borel measures on $\mathbb{R}^{n}$. Denote also by $C_{0}(\mathbb{R}^{n})$ and $C_{0}(\mathbb{R}^{n}, X)$  the spaces of continuous and  $X$-valued continuous functions, respectively, vanishing at $\infty$. Due to the Radon-Nikod\'{y}m property of $X$ and the fact that $C_{0}(\mathbb{R}^{n})$ has the approximation property (which follows from the fact that it has a Schauder basis \cite[Corollary~4.1.4, p.~112]{S82}), we have the following natural isomorphisms, 
\[
\mathcal{M}_{1}(\mathbb{R}^{n}, X)\cong \mathcal{M}(\mathbb{R}^{n})\hat{\otimes}_{\pi} X \cong ( C_{0}(\mathbb{R}^{n}) \hat{\otimes}_{\varepsilon} Y )',
\]
where we have used \cite[Theorem 5.22, p.~108]{ryan2002} in the first isomorphism and \cite[Theorem 5.33, p.~114]{ryan2002} in the second one, and we recall that the symbols $\hat{\otimes}_{\pi}$ and $\hat{\otimes}_{\varepsilon}$  stand for the projective and injective completed tensor products. Also, reasoning as in \cite[Example~3.3, p.~47]{ryan2002}, one readily verifies that  $C_{0}(\mathbb{R}^{n}) \hat{\otimes}_{\varepsilon} Y= C_{0}(\mathbb{R}^{n}, Y)$. Summarizing, we may view $\mathcal{M}_{1}(\mathbb{R}^{n}, X)$ as the dual of $C_{0}(\mathbb{R}^{n},Y)$. Similarly as in \emph{Case I}, we obtain with the aid of the Banach-Alaoglu theorem that there is an $X$-valued vector measure $\boldsymbol{\mu}$ such that $\mathbf{f}=\mathrm{d} \boldsymbol{\mu}$. Since $X$ has the Radon-Nikod\'{y}m property, we must prove that $\boldsymbol{\mu}$ is absolutely continuous with respect to the Lebesgue measure in order to show that $\mathbf{f}\in L^{1}(\mathbb{R}^{n}, X)$. By the Hanh-Banach theorem it suffices to show that if $\mathbf{w}^{\ast}\in X$, then the scalar valued measure $\mu_{\mathbf{w}^{\ast}}=\langle\mathbf{w}^{\ast}, \boldsymbol{\mu}\rangle$ is absolutely continuous with respect to the Lebesgue measure. But note that $\mathrm{d} \mu_{\mathbf{w}^{\ast}}$ is the distributional boundary value of $\langle \mathbf{w}^{\ast} ,\mathbf{F}\rangle $. The function $\langle \mathbf{w}^{\ast} ,\mathbf{F}\rangle $ clearly belongs to the scalar valued Hardy space $H^{1}(T^{C})$ and the well-known classical result \cite[Theorem~5.6, p.~119]{S-W71} says that it has boundary value in $L^{1}(\mathbb{R}^{n})$, so that indeed $\mu_{\mathbf{w}^{\ast}}=\langle \mathbf{w}^{\ast}, \boldsymbol{\mu}\rangle$ is absolutely continuous with respect to the Lebesgue measure.
\end{proof}

Let us point out that any reflexive Banach space has the Radon-Nikod\'{y}m
 property \cite[Corollary~4, p.~82]{D-U77}, whence we immediately obtain the ensuing corollary.
 
 \begin{corollary}
 \label{bvHardyCor} Let $X$ be a reflexive Banach space, let $C\subset \mathbb{R}^{n}$ be a regular open convex cone, and let $p\in [1,\infty]$. Any $\mathbf{F}\in H^{p}(T^{C},X)$ admits a Poisson integral representation \eqref{bvHardyeq2} for some $\mathbf{f}\in L^{p}(\mathbb{R}^{n},X)$.
 \end{corollary}


\begin{thebibliography}{99}

\bibitem{C-P-V2019} Carmichael, R.~D., Pilipovi\'{c}, S, Vindas, J.: Vector valued Hardy spaces related to analytic functions having distributional boundary values. Complex Var. Ellip. Equ. \textbf{64}, 1634--1654 (2019)

\bibitem{C-R83}  Carmichael, R., Richters, S.: Holomorphic functions in tubes which have distributional boundary values and which are $H^{p}$ functions. SIAM J. Math. Anal. \textbf{14}, 596--621 (1983)

\bibitem{D-U77} Diestel, J., Uhl, J.~: Vector measures. American Mathematical Society, Providence, RI (1977)

\bibitem{HormanderVolI} H\"ormander, L: The analysis of linear partial differential operators. I. Distribution
theory and Fourier analysis. Springer-Verlag, Berlin (1990)


\bibitem{Raina} Raina, A.: On the role of Hardy spaces in form factor bounds. Lett. Math. Phys. \textbf{2}, 513--519 (1978)

\bibitem{Rudin71} Rudin, W.: Lectures on the edge-of-the-wedge theorem. Vol. 6. Conference board of the
mathematical sciences regional conference series in mathematics. American Mathematical Society, Providence, RI (1971)

\bibitem{ryan2002} Ryan, R.~A.: Introduction to tensor products of Banach spaces. Springer-Verlag London, Ltd., London (2002)

\bibitem{S82} Semadeni, Z.: Schauder bases in Banach spaces of continuous functions. Lecture Notes in Mathematics, vol. 918. Springer-Verlag, Berlin-New York (1982)

\bibitem{S-W71}  Stein, E., Weiss, G.: Introduction to Fourier analysis on Euclidean spaces. Princeton University Press,  Princeton, NJ (1971)

\bibitem{T67} Tr\`{e}ves, F.: Topological vector spaces, distributions and kernels. Academic Press, New York-London (1967)




\end{thebibliography}
\end{document}